\documentclass[12pt,reqno]{amsart}
\usepackage[top=1in, bottom=1in, left=1in, right=1in]{geometry}
\usepackage{hyperref}
\usepackage[alphabetic,backrefs]{amsrefs}
\usepackage{caption}
\usepackage{times, amsthm, amssymb, amsmath, amsfonts, graphicx, mathrsfs}
\usepackage{mathtools}
\usepackage{tikz} 
\usetikzlibrary{arrows, cd, matrix,positioning}
\usepackage{xcolor}
\usepackage{bbm}
\usepackage{mathscinet}
\usepackage{listings} 

\newtheorem{theorem}{Theorem}[section]

\theoremstyle{definition}
\newtheorem{definition}[theorem]{Definition}

\theoremstyle{remark}
\newtheorem{remark}[theorem]{Remark}

\newcommand{\bbQ}{\mathbbm{Q}}

\newcommand{\coordsys}[4]{
\foreach \x in {#1,...,#3}
\draw[line width=.6pt,color=black!15,dashed] (\x,#2-.2) -- (\x,#4+.2);
\foreach \y in {#2,...,#4}
\draw[line width=.6pt,color=black!15,dashed] (#1-.2,\y) -- (#3+.2,\y);
\draw[->] (#1-.2,0) -- (#3+.2,0) node[right] {$a$};
\draw[->] (0,#2-.2) -- (0,#4+.2) node[above] {$b$};
\foreach \x in {#1,...,-1}
\draw[shift={(\x,0)},color=black] (0pt,2pt) -- (0pt,-2pt) node[below] {\footnotesize $\x$};
\foreach \x in {1,...,#3}
\draw[shift={(\x,0)},color=black] (0pt,2pt) -- (0pt,-2pt) node[below] {\footnotesize $\x$};
\foreach \y in {#2,...,-1}
\draw[shift={(0,\y)},color=black] (2pt,0pt) -- (-2pt,0pt) node[left] {\footnotesize $\y$};
\foreach \y in {1,...,#4}
\draw[shift={(0,\y)},color=black] (2pt,0pt) -- (-2pt,0pt) node[left] {\footnotesize $\y$};
}

\definecolor{codegreen}{rgb}{0,0.6,0}
\definecolor{codegray}{rgb}{0.5,0.5,0.5}
\definecolor{codepurple}{rgb}{0.58,0,0.82}
\definecolor{backcolour}{rgb}{0.95,0.95,0.92}

\lstdefinelanguage{m2}{
    morekeywords={
        chainComplex,
        monomialIdeal,
        prune,
        needsPackage,
        matrix,
        hashTable,
        numColumns,
        lift,
        apply,
        ring,
        ideal,
        dim,
        product,
        numgens,
        max,
        assert,
        betti,
        res,
        cokernel
        },
    morekeywords=[2]{
        HH,
        restart,
        false,
        true,
        id
    },
    morekeywords=[3]{
        ZZ,
        QQ
    },
    morekeywords=[4]{
        of,
        for,
        in,
        do
    }
}

\lstdefinestyle{mystyle}{
    backgroundcolor=\color{backcolour},   
    commentstyle=\color{codegreen},
    keywordstyle=\color{magenta},
    keywordstyle=[2]\color{blue},
    keywordstyle=[3]\color{green},
    keywordstyle=[4]\color{teal},
    numberstyle=\tiny\color{codegray},
    stringstyle=\color{codepurple},
    basicstyle=\ttfamily\footnotesize,
    breakatwhitespace=false,
    breaklines=true,
    captionpos=b,
    keepspaces=true,
    numbers=none,
    numbersep=5pt,
    showspaces=false,
    showstringspaces=false,
    showtabs=false,
    tabsize=2
}

\title{CellularResolutions M2 Package}
\author{Aleksandra Sobieska}
\address{Department of Mathematics, University of Wisconsin -- Madison, Madison, WI 53706}
\email{asobieska@math.wisc.edu}
\urladdr{\href{https://people.math.wisc.edu/~asobieska}{https://www.people.math.wisc.edu/~asobieska/}}

\author{Jay Yang}
\address{Department of Mathematics, Washington University in St. Louis, St. Louis, MO 63130}
\email{jayy@wustl.edu}
\urladdr{\href{https://www.math.wustl.edu/~jayy/}{https://www.math.wustl.edu/~jayy/}}
\date{\today}

\begin{document}


\maketitle

\begin{abstract}
Cellular resolutions are a technique for constructing resolutions of monomial ideals by giving a cell complex labeled by monomials, or more generally, by monomial modules. This \verb|Macaulay2| package allows us to work with cellular resolutions in a natural way.
\end{abstract}

\section{Introduction}

Simplicial resolutions were originally introduced in \cite{monomial-res} by Bayer, Peeva, and Sturmfels, and later generalized to cellular resolutions in \cite{cellular-res}. Further generalizations are also possible as in \cite{clark}. Since their introduction, they have been both a unifying theme theme, giving a geometric realization to previously disparate resolutions of monomial ideals, including the Taylor resolution \cite{taylor}, Lyubeznik resolution \cite{lyubeznik}, and Eliahou--Kervaire resolution \cite{mermin-ekres}, as well as a tool to construct new resolutions, for example, those supported by the Scarf complex \cite{scarf} and the hull complex \cite{cellular-res}.

Prior to this package, Macaulay2\cite{m2} had some support for simplicial resolutions via the\linebreak[4] \verb|SimplicialComplexes| package\cite{m2-simplicial-complexes}. However, \verb|SimplicialComplexes| focuses primarily on the Stanley-Reisner correspondence and therefore simplicial complexes; our goal is to offer an interface to work specifically with cellular complexes and cellular resolutions, allowing natural constructions and manipulation of these complexes, their labels, and the resolutions they support.

Furthermore, while the standard theory of cellular resolutions considers labels by monomials, this constraint is not needed for the development of the theory, and in fact removing such constraint can be useful for certain applications \cites{miller, yang}. In our package, we allow labels by monomials, monomial ideals, and, more generally, monomial submodules of a free module. The primary constraint is that there must be natural inclusions from the label of a cell to the labels of the cells in its boundary.

In shedding the restriction to cellular resolutions, we lose the capacity to check that a purported cell complex truly is a cell complex in the topological sense. However, this does not impair the functioning of the package. To wit, the package can be thought of as acting upon certain labeled posets. Furthermore, in nearly all cases, these will in fact be cellular, and the structure of the commands in our package encourages the creation of geometrically realizable cell complexes.

In the same vein, we do not actually store sufficient information to uniquely identify a geometric realization for our cell complexes. Instead, sufficient information is retained to preserve the inclusion structure of the complex, along with its homological properties.

As a consequence of these limitations and generalizations, the class of ``cellular resolutions'' captured by our package are of a slightly odd form, although they capture essentially all common cases.

\begin{definition}
A \emph{combinatorial cell complex} $\Delta$ is a collection of cells in non-negative dimensions such that each pair of cells $C, D\in \Delta$, with $\dim D + 1 = \dim C$ has an attaching degree $\alpha(C,D)~\in~\mathbb{Z}$. such that for every pair $A,C\in\Delta$ with $\dim A = \dim C + 2$
\[\sum_{\substack{B\in\Delta \\ \dim B = \dim C + 1}} \alpha(A,B)\alpha(B,C) = 0 \]
and if $\dim C=1$ then there are either 2 or 0 cells $D\in\Delta$ such that $\alpha(C,D)\neq 0$ and they have to be $\pm 1$.
\end{definition}

\begin{remark}
The above definition is somewhat non-standard, and, in particular, a combinatorial cell complex entirely forgets any topological data about the attaching map. However, at the level of cellular resolutions, that data is unimportant and, with the loss of topological specifications, combinatorial cell complexes need not represent actual topological cell complexes. In fact, one can construct combinatorial cell complexes for which no topological realization exists.
\end{remark}

This definition then allows for the definition of a chain complex and homology for a combinatorial cell complex. In particular, for any ring $R$, the $r$-chains are given by formal $R$-linear combinations of cells of dimension $r$ and the boundary map is given by
\[\partial([C]) = \sum_{\substack{D\in\Delta\\ \dim D + 1 = \dim C}} \alpha(C,D)[D].\]

This gives us homology of the cell complex over $R$, and now we can consider labeled cell complexes and the associated homology. This in itself is simply a reiteration of the standard topological theory of cell complexes. Next we define the main object of concern for this package.

\begin{definition}
Given a polynomial ring $S$, a \emph{labeled cell complex} is a tuple $(X,\{M_C\},\{i_{C,D}\})$ where $X$ is a combinatorial cell complex, $M_C$ is an $S$-module for each cell $C\in X$, and \\$i_{C,D}~:~M_C~\rightarrow~M_D$ is an inclusion such that 
\[i_{C,D}\circ i_{E,C} = i_{C',D}\circ i_{E,C'}\]
whenever $\alpha(C,D)\neq 0$.
\end{definition}

\begin{definition}[\cite{cellular-res}]
Let $F_{\bullet}(X)$ be the chain complex given by
\[F_{i}(X) = \bigoplus_{\substack{C\in X \\ \dim C = i}} M_C\]
\[\partial|_{M_C} = \sum_{D\in \partial C} \alpha(C,D) i_{C,D}\]
where $\alpha(C,D)$ is the attaching degree of the cell $D$ in $C$ and $i_{C,D}:M_{C}\rightarrow M_{D}$ is the canonical inclusion.
\end{definition}

Now the core theorem for cellular resolutions relates the homology of this chain complex to the homology over the coefficient ring of subcomplexes of the original simplicial complex.

\begin{theorem}
\label{thm:homology-general}
    Let $X$ be a labeled cell complex with labels given by multigraded sub-modules of a common multigraded module over a $k$-algebra for a field $k$ and inclusions given by maps of degree $\mathbf{0}$. For a multidegree $\mathbf{b}$, let $X_{\preceq \mathbf{b}}$ be the labeled subcomplex formed by the cells $C$ with $(M_C)_\mathbf{b}\neq 0$ where the labels are given by the $k$-vector space $(M_C)_\mathbf{b}$. Then 
    \[H_i(F_\bullet(X))_{\mathbf{b}} = H_i(X_{\preceq \mathbf{b}}).\]
\end{theorem}
\begin{proof}
   Notice that $H_i(F_\bullet(X))_\mathbf{b}$ is computed precisely by the complex of $k$-vector spaces given by
   $(F_\bullet(X))_\mathbf{b}$, the subcomplex given by the degree $\mathbf{b}$ components of $F_\bullet(X)$.
   But this is exactly the chain complex given by taking $X_{\prec \mathbf{b}}$ with the labeling described in the statement of the theorem. Thus we have
    \[H_i(F_\bullet(X))_{\mathbf{b}} = H_i(X_{\preceq \mathbf{b}}),\]
    as desired.
\end{proof}

The primary use case is when $X$ is labeled by free submodules of a fine-graded polynomial ring. These are in one-to-one correspondence with the monomials, so we recover the standard formulation of a cell complex $X$ labeled by monomials $m_C$ satisfying the condition $m_D \mid m_C$ for $\alpha(C,D) \neq 0$ by simply using the labels given by the free module generated in the fine-graded multidegree of $m_C$.

Moreover, in the case of submodules of a fine-graded polynomial ring over a field, the vector space of degree $\alpha$ polynomials is always either $0$ or $1$ dimensional, so the resulting homology in Theorem~\ref{thm:homology-general} reduces exactly to the standard cellular homology, which allows for the following restatement of the classic result about cellular resolutions.

\begin{theorem}[\cite{cellular-res}]
\label{thm:resolutions}
    Let $X$ be a labeled cell complex with labels given by monomials $m_C$ for cells $C\in X$ satisfying the condition $m_D \mid m_C$ for $\alpha(C,D)\neq 0$. For a monomial $m$, let $X_{\preceq m}$ be the subcomplex formed by the cells $C$ with $m_C \mid m$. If $X_{\preceq m}$ is acyclic for each monomial $m$, then $F_{\bullet}(X)$ is a free resolution of the ideal $\left<m_v \ | \ v \in X(0) \right>$. 
    
    Moreover if $m_C \neq m_D$ for each pair of cells $C$ and $D$ where $\alpha(C,D) \neq 0$, then the resolution $F_{\bullet}(X)$ is also minimal.
\end{theorem}

Note that for this to be a free resolution, the labels must be free modules. For the remainder of this paper we will primarily focus on the case of free module labels, but the code does not require this.

The final standard modification to this construction is to instead use the augmented chain complex, and thus give a resolution of the module $R/I$ as opposed to the ideal $I$ as a module. For a CW complex this is usually accomplished by placing a copy of the field at homological degree $-1$ in the chain complex with an augmenting morphism from the $0$ chains. 

In the case of a labeled cell complex, the story is only slightly more subtle, but extending from the simplicial case, the answer becomes clear. In the simplicial setting, where reduced homology is the more natural homology, there is a $-1$-simplex corresponding to the empty face, and the label on this simplex gives the module in homological degree $-1$. And with submodules labeling the simplices, the usual label on the empty face is the ambient module. As such, our extension to the cellular case proceeds in the same manner, and for reduced homology, we add a copy of the ambient module in homological degree $-1$ to the chain complex.

\section{Usage}

Consider the monomial ideal $I = \langle yw, xyz, x^2y, z^4w \rangle$ in $S = \bbQ[x,y,z,w]$. We could opt for several different cell complexes supporting a resolution of $I$. 

\begin{figure}[!h] 
    \centering
    \begin{minipage}{0.45\textwidth}
    \centering
    \begin{tikzpicture}[scale=4]
	\tikzstyle{vertex}=[circle,thick,draw=black,fill=black,inner sep=0pt,minimum width=4pt,minimum height=4pt]
	
	\coordinate (A) at (0,0);
	\coordinate (B) at (1,-0.25); 
	\coordinate (C) at (1.25,0.25) {};
	\coordinate (D) at (0.75,1) {}; 
	
	\draw[fill=gray, opacity=0.5] (A) -- (B) -- (C) -- (D); 
	
	\node[vertex, label=left:$x^2y$] (m1) at (A) {}; 
	\node[vertex, label=below:$xyz$] (m2) at (B) {}; 
	\node[vertex, label=right:$yw$] (m3) at (C) {}; 
	\node[vertex, label=above:$z^4w$] (m4) at (D) {}; 
	
	\draw (m1) -- (m2) 
	    (m1) -- (m3)
	    (m1) -- (m4)
	    (m2) -- (m3)
	    (m2) -- (m4)
	    (m3) -- (m4); 
	
	\end{tikzpicture}
	\caption{Cell complex $\Delta_T$ supporting the Taylor resolution}
    \label{fig:tetrahedron}
	\end{minipage}
	\hfill
    \begin{minipage}{0.45\textwidth}
    \centering
    \begin{tikzpicture}[scale=4] 
	\tikzstyle{vertex}=[circle,thick,draw=black,fill=black,inner sep=0pt,minimum width=4pt,minimum height=4pt]
	
	\coordinate (A) at (0,0);
	\coordinate (B) at (1,-0.25); 
	\coordinate (C) at (1.25,0.25) {};
	\coordinate (D) at (0.75,1) {}; 
	
	\coordinate (AB) at (0.5,-0.125);
	\coordinate (AC) at (0.625,0.125);
	\coordinate (BC) at (1.125,0);
	\coordinate (CD) at (1,0.625); 
	\coordinate (ABC) at (0.75,-0.125);
	
	\draw[fill=gray, opacity=0.5] (A) -- (B) -- (C); 
	
	\node[vertex, label=left:$x^2y$] (m1) at (A) {}; 
	\node[vertex, label=below:$xyz$] (m2) at (B) {}; 
	\node[vertex, label=right:$yw$] (m3) at (C) {}; 
	\node[vertex, label=above:$z^4w$] (m4) at (D) {}; 
	
	\draw (m1) -- (m2) 
	    (m1) -- (m3)
	    (m2) -- (m3)
	    (m3) -- (m4); 
	    
	\node[label=below:$x^2yz$] (m12) at (AB) {};
	\node[label=above:$x^2yw$] (m13) at (AC) {}; 
	\node[label=right:$xyzw$] (m23) at (BC) {}; 
	\node[label=right:$yz^4w$] (m34) at (CD) {}; 
	\node[label=$x^2yzw$] (m123) at (ABC) {}; 
	
	\end{tikzpicture}
	
	\caption{Cell complex $\Delta$}
    \label{fig:tetrahedron2}
	\end{minipage}
\end{figure}

The most straightforward cellular resolution is the Taylor resolution, which is supported on the full $4$-simplex whose vertices are labeled with the generators of $I$. The underlying cell complex $\Delta_T$, pictured in Figure~1, can be easily constructed using the \texttt{taylorComplex} command.

\begin{lstlisting}[language=m2,style=mystyle]
i: S = QQ[x,y,z,w];
i: I = monomialIdeal (y*w, x*y*z, x^2*y, z^4*w);
i: DeltaT = taylorComplex I
o= DeltaT
o: CellComplex
\end{lstlisting}

From any cell complex, one obtains the corresponding chain complex with the \texttt{chainComplex} command. 

\begin{lstlisting}[language=m2,style=mystyle]
i: T = chainComplex DeltaT
o: S^1 <-- S^4 <-- S^6 <-- S^4 <-- S^1 
   -1      0       1       2       3
o= ChainComplex
\end{lstlisting}

The minimal resolution is supported on the complex $\Delta$, pictured in Figure~2. We construct $\Delta$ and the free resolution $F$ supported on $\Delta$ in our package in the following way. 

\begin{lstlisting}[language=m2,style=mystyle]
i: v1 = newCell({},y*w);
i: v2 = newCell({},x*y*z);
i: v3 = newCell({},x^2*y);
i: v4 = newCell({},z^4*w);
i: e12 = newCell({v1,v2});
i: e13 = newCell({v1,v3});
i: e23 = newCell({v2,v3});
i: e14 = newCell({v1,v4});
i: f123 = newCell({e12,e13,e23});
i: Delta = cellComplex(S, {f123,e14})
o= Delta
o: CellComplex
i: C = chainComplex Delta 
o= S^1 <-- S^4 <-- S^4 <-- S^1
   -1      0       1       2
o: ChainComplex
\end{lstlisting}

The general syntax to create a cell uses the \texttt{newCell(b,a)} command, where \texttt{b} is a list of boundary cells and \texttt{a} is the cell label. If \texttt{a} is omitted when the boundary cells have monomial labels, then the label for the cell is inferred by taking the least common multiple of its boundary labels. In particular, the default label for a cell with the empty boundary (a vertex in the simplicial case) is $1$.  Note that the chain complex $C$ includes a copy of $S$ in the $-1$ spot, as it computes reduced homology. 

This package extends the functor \texttt{HH} to compute the homology of the chain complex directly from the cell complex. We can also check minimality of the resolutions directly from the cell complexes using \texttt{isMinimal}. Here we apply \texttt{HH} and \texttt{isMinimal} to $\Delta$ and $\Delta_T$ to find that both $F$ and $T$ are resolutions of $S/I$, but that $F$ is minimal while $T$ is not. 

\begin{lstlisting}[language=m2,style=mystyle]
i: prune HH C
o= -1 : cokernel | yw xyz x2y z4w |
    0 : 0 
    1 : 0 
    2 : 0  
o: GradedModule
i: prune HH Delta
o= -1 : cokernel | yw xyz x2y z4w |
    0 : 0 
    1 : 0 
    2 : 0  
o: GradedModule
i: isMinimal DeltaT 
o= false 
i: isMinimal Delta 
o= true
\end{lstlisting}

Should the user find the automatic inclusion of the augmentation map in the chain complex troubling, one can specify that the augmentation map be omitted by setting the \texttt{Reduced} option in \texttt{chainComplex} to \texttt{false}. 

\begin{lstlisting}[language=m2,style=mystyle]
i: Cnonred = chainComplex(Delta, Reduced => false)
o= S^4 <-- S^4 <-- S^1        
   0       1       2
o: ChainComplex
\end{lstlisting}

Two prominent complexes in the study of cellular resolutions are the hull and Scarf complexes, also implemented. In this case, the hull complex produces the full simplicial complex and the Scarf complex is minimal and exact.

\begin{lstlisting}[language=m2,style=mystyle]
i: H = hullComplex I 
o= H
o: CellComplex
i: scarf = scarfComplex I
o= scarf
o: CellComplex
i: prune HH scarf
o= -1 : cokernel | yw xyz x2y z4w |
    0 : 0
    1 : 0
    2 : 0
o: GradedModule
\end{lstlisting}

The examples given are perhaps disingenuous -- it is possible for the hull complex to better capture the nature of our ideal. Here we recreate the hull complex from  Example~4.23 from \cite{miller-sturmfels-cca}, which supports a minimal resolution. The corresponding cell complex is shown in Figure~\ref{fig:hull}.

\begin{lstlisting}[language=m2,style=mystyle]
i: H2 = hullComplex monomialIdeal {x^2*z, x*y*z, y^2*z, x^3*y^5, x^4*y^4, x^5*y^3}
o= H2
o: CellComplex
i: isMinimal H2 
o= true
\end{lstlisting}

\begin{center}
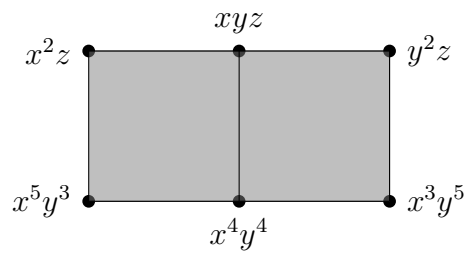
\begin{figure}[h]
\begin{tikzpicture}[scale=2]
	\tikzstyle{vertex}=[circle,thick,draw=black,fill=black,inner sep=0pt,minimum width=4pt,minimum height=4pt]

    \coordinate (A) at (0,0);
    \coordinate (B) at (1,0);
    \coordinate (C) at (2,0);
    \coordinate (D) at (0,1);
    \coordinate (E) at (1,1);
    \coordinate (F) at (2,1);

	\node[vertex, label=left:$x^2z$] (x2z) at (D) {}; 
    \node[vertex, label=above:$xyz$] (xyz) at (E) {};
    \node[vertex, label=right:$y^2z$] (y2z) at (F) {};
    \node[vertex, label=right:$x^3y^5$] (x3y5) at (C) {};
    \node[vertex, label=below:$x^4y^4$] (x4y4) at (B) {};
    \node[vertex, label=left:$x^5y^3$] (x5y3) at (A) {};

    \draw[fill=gray, opacity=0.5] (A) -- (C) -- (F) -- (D);

    \foreach \x\y in {A/B, B/C, D/E, E/F, A/D, B/E, C/F} \draw (\x) -- (\y);

\end{tikzpicture}
\caption{The hull complex $H_2$ supporting a minimal resolution of $\langle x^2z, xyz, y^2z, x^3y^5, x^4y^4, x^5y^3 \rangle$}
\label{fig:hull}
\end{figure}
\end{center}

The Scarf complex, on the other hand, may not be exact, but when it does support a resolution, it is a minimal one. In particular, the Scarf complex supports a resolution for generic monomial ideals \cite{monomial-res}*{Theorem~3.2}. Below is an example where the Scarf complex fails to be exact.

\begin{lstlisting}[language=m2,style=mystyle]
i: scarf2 = scarfComplex monomialIdeal {x*y, y*z, z*w, w*x}
o= scarf2
o: CellComplex
i: prune HH scarf2 
o= -1 : cokernel | zw xw yz xy |
    0 : 0
    1 : S^1
o: GradedModule
\end{lstlisting}


While it is not the focus of the package, \verb|CellularResolutions| can also be used to manipulate cell complexes, independent of their supported resolutions. In particular, we have functions \texttt{cellComplexSphere}, \texttt{cellComplexRPn}, \texttt{cellComplexTorus} which construct cell complexes for some common topological spaces. The arguments are the ambient ring for the labels and the dimension of the desired complex. The labeling is the one induced by setting all vertex labels to $1$. The code below also demonstrates the \texttt{cells} command, which returns a hashtable of the cells separated by dimension. 

\begin{lstlisting}[language=m2,style=mystyle]
i: S2 = cellComplexSphere(QQ,2)
o= S2
o: CellComplex
i: cells(S2) 
o= HashTable{0 => {Cell of dimension 0}}
             2 => {Cell of dimension 2}
o: HashTable
i: prune HH S2
o= -1 : 0  
    0 : 0  
    1 : 0
    2 : QQ^1
o: GradedModule
i: T3 = cellComplexTorus(QQ,3)
o= T3
o: CellComplex
i: prune HH T3 
o= -1 : 0 
    0 : 0  
    1 : QQ^3
    2 : QQ^3
    3 : QQ^1
o: GradedModule
i: Z2 = ZZ/2;
i: RP3 = cellComplexRPn(Z2,3)
o= RP3
o: CellComplex
i: prune HH RP3
o= -1 : 0 
    0 : 0  
    1 : Z2^1
    2 : Z2^1
    3 : Z2^1
o: GradedModule
\end{lstlisting}


Associated to our cell complex object is an underlying poset with respect to cell inclusion. This poset can be accessed with the \texttt{facePoset} command. 

\begin{lstlisting}[language=m2,style=mystyle]
i: facePoset RP3
o= Relation Matrix: | 1 1 1 1 |
                    | 0 1 1 1 |
                    | 0 0 1 1 |
                    | 0 0 0 1 |
o: Poset
\end{lstlisting}

Oftentimes cellular resolutions arise in the context of polyhedra, so the package allows for conversion of polyhedra or polyhedral complexes to cellular complexes. For example, we can easily create a tetrahedral cell complex from the polyhedral version. 

\begin{lstlisting}[language=m2,style=mystyle]
i: P = convexHull id_(ZZ^4);
i: Pcellular = cellComplex(S,P)
o= Pcellular
o: CellComplex
\end{lstlisting}

The default labeling for a cell complex arising from a polyhedron or polyhedral complex is a $1$ at every vertex. One can obtain more sophisticated labeling and relabeling in a few ways.

The relabel command takes as input a cell complex and a hashtable whose keys are the vertices of the complex with values corresponding to the new labels. For example, from the tetrahedral cell complex above, we can recreate the Taylor resolution of the monomial ideal $\langle yw, xyz, x^2y, z^4w \rangle$. 

\begin{lstlisting}[language=m2,style=mystyle]
i: verts = cells(0,Pcellular);
i: H = hashTable {verts#0 => y*w, verts#1 => x*y*z, verts#2 => x^2*y, verts#3 => z^4*w};
i: relabeledP = relabelCellComplex(Pcellular, H)
o= relabeledP
o: CellComplex
i: apply(cells(0,relabeledP),cellLabel)
o= {y*w, x^2*y, x*y*z, z^4*w}
o: List
\end{lstlisting}

\texttt{Labels} is an option available when creating a cell complex from a polyhedron or polyhedral complex; the corresponding polyhedral and cell complexes are shown in Figure~\ref{fig:relabel}.

\begin{lstlisting}[language=m2,style=mystyle]
i: R = QQ[a,b];
i: P1 = convexHull matrix {{5,3},{1,2}};
i: P2 = convexHull matrix {{3,2},{2,3}};
i: P3 = convexHull matrix {{2,0},{3,7}};
i: PC = polyhedralComplex {P1,P2,P3};
i: v = vertices PC;
i: H = hashTable for i to (numColumns v - 1) list (v_i => a^(lift(v_i_0,ZZ))*b^(lift(v_i_1,ZZ)));
i: PCcellular = cellComplex(R,PC, Labels => H);
i: apply(cells(0,PCcellular),cellLabel)
o= {a^3b^2 , a^2b^3 , a^5b, b^7}
o: List
i= apply(cells(1,PCcellular),cellLabel)
o: {a^5b^2, a^3b^3, a^2b^7}
o: List
\end{lstlisting}

\begin{center}
\begin{figure}
\begin{tikzpicture}[scale=0.75, vertices/.style={draw, fill=black, circle, inner sep=2pt}]
    \coordsys{-1}{-1}{8}{8}
    \node (m1) at (5,1) [vertices] {};
    \node (m2) at (3,2) [vertices] {};
    \node (m3) at (2,3) [vertices] {};
    \node (m4) at (0,7) [vertices] {};

    \draw [ultra thick, red] (m1) -- (m2);
    \draw [ultra thick, blue] (m2) -- (m3);
    \draw [ultra thick, purple] (m3) -- (m4);

    \node (m12) at (4,1.5) [label=above:\textcolor{red}{$P_1$}] {};
    \node (m23) at (2.5,2.5) [label=below left:\textcolor{blue}{$P_2$}] {};
    \node (m34) at (1,5) [label=right:\textcolor{purple}{$P_3$}] {};
\end{tikzpicture}
\hspace{0.75cm}
\begin{tikzpicture}[vertices/.style={draw, fill=black, circle, inner sep=2pt}]
    \node (m1) at (0,0) [vertices, label=$a^5b$] {};
    \node (m2) at (2,0) [vertices, label=$a^3b^2$] {};
    \node (m3) at (4,0) [vertices, label=$a^2b^3$] {};
    \node (m4) at (6,0) [vertices, label=$b^7$] {}; 

    \draw [ultra thick, red] (m1) -- (m2);
    \draw [ultra thick, blue] (m2) -- (m3);
    \draw [ultra thick, purple] (m3) -- (m4);

    \node (m12) at (1,0) [label=below:\textcolor{red}{$a^5b^2$}] {};
    \node (m23) at (3,0) [label=below:\textcolor{blue}{$a^3b^3$}] {};
    \node (m34) at (5,0) [label=below:\textcolor{purple}{$a^2b^7$}] {};

    \node (spacer) at (3,-4) {};
\end{tikzpicture}
\caption{A polyhedral complex and corresponding relabeled cell complex}
\label{fig:relabel}
\end{figure}
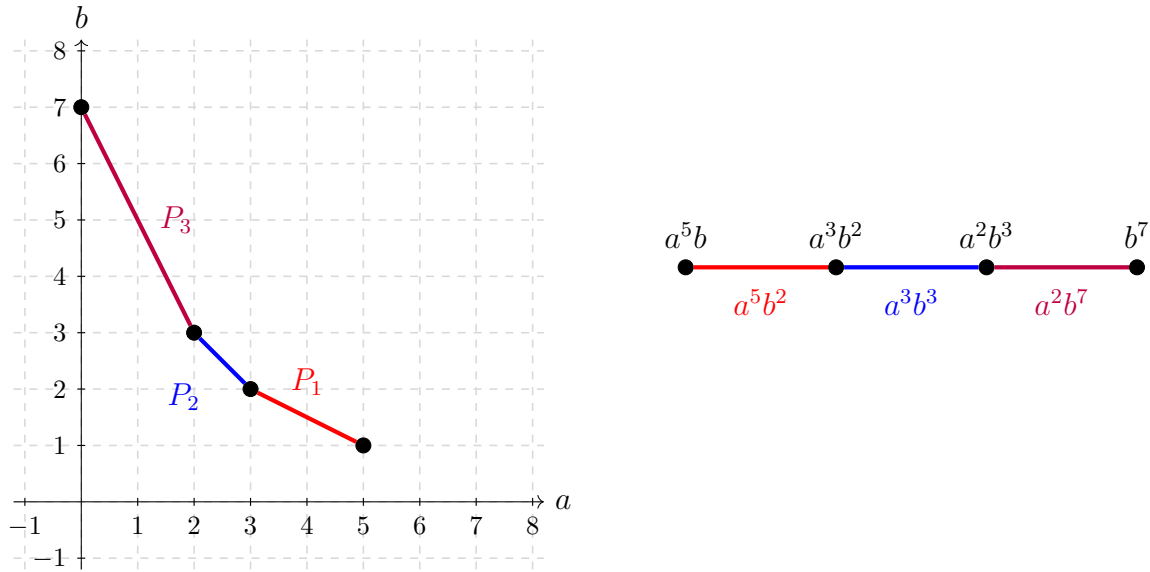
\end{center}



As a more involved example, let us consider the question of the minimal free resolution of the irrelevant ideal for a toric variety. Explicitly, we will consider the case of $\mathbb{P}^1\times \mathbb{P}^2$. In general for a projective toric variety, we can use the polytope associated to the variety to construct a cellular resolution for the irrelevant ideal. In this case our ring is given by $S=k[x_0,x_1,x_2,x_3,x_4]$ and the irrelevant ideal $B$ is $\left<x_0,x_1\right>\cap \left<x_2,x_3,x_4\right> = \left<x_0x_2,x_0x_3,x_0x_4,x_1x_2,x_1x_3,x_1x_4\right>$.

\begin{lstlisting}[language=m2, style=mystyle]
i: needsPackage "NormalToricVarieties"
i: X = toricProjectiveSpace(1) ** toricProjectiveSpace(2);
i: S = ring X;
i: B = ideal X;
i: Sigma = fan X;
i: P = polytope Sigma;
i: d = dim P;
i: F = faces_d P;
i: m = product(apply(numgens S, i -> S_i));
i: G = apply(max X, l -> m//product(apply(l,i -> S_i)));
i: H = hashTable apply(#G, i -> (j := F#i#0#0;((vertices P)_j,G_i)))
i: C = cellComplex(S,P,Labels => H);
i: Cres = (chainComplex C)[-1]
o: S^1  <-- S^6  <-- S^9  <-- S^5  <-- S^1
   0        1        2        3        4
o: ChainComplex
i: assert(betti (res B) == betti Cres)
i: assert(HH_0 Cres == S^1/B)
i: for i in {1,2,3} do assert(HH_i Cres == 0)
\end{lstlisting}

The homological shift in the chain complex \texttt{Cres} shifts the resolution to resolve the quotient $S/B$, since the package default is to place modules corresponding to $d$-cells in homological degree $d$. 




\section{Acknowledgement}

The authors would like to thank the organizers of the Macaulay2 Workshop in Minneapolis in 2019, where work on this paper and package began.

In addition, this work is based in part on work supported by NSF RTG grant 1745638 while JY was a postdoc at Minnesota and in part upon work supported by the NSF grant 1440140, while JY was in residence at the Mathematical Sciences Research Institute in Berkeley, California, during the fall semester of 2023.

\end{document}